%% file: Articolo_parabole.tex
\documentclass[11pt]{CGC2}
\usepackage[english]{babel}
\usepackage{verbatim}

\usepackage{psfrag}
\usepackage{graphicx}
\usepackage{textcomp}
\usepackage{xcolor}
\usepackage{makeidx}  

\usepackage{amsmath}
\usepackage{amsmath}
\usepackage{amsfonts}
\usepackage{amssymb}
\usepackage{mathrsfs}
\usepackage{algorithmic}
\usepackage{algorithm}

\usepackage{pstricks}
\usepackage{pifont}
\usepackage{latexsym}
\usepackage{booktabs}
\usepackage{array}
\usepackage{color}
\usepackage{colortbl}

\theoremstyle{theorem}
\newtheorem*{teo}{Theorem A}

\newcommand{\F}{\mathbb{F}_2}
\newcommand{\FQ}{\mathbb{F}_{q^2}}
\newcommand{\Fq}{\mathbb{F}_q}
\newcommand{\Tr}[3]{\mathrm{Tr}^{#1}_{#2}(#3)}
\newcommand{\Tra}{\mathrm{Tr}}
\newcommand{\N}[3]{\mathrm{N}^{#1}_{#2}(#3)}
\newcommand{\No}{\mathrm{N}}
\newcommand{\He}{\mathcal{H}}
\newcommand{\I}{\mathrm{Im}}
\newcommand{\xx}{\mathcal{X}}
\newcommand{\Bmq}{\mathcal{B}_{m,q}}

\begin{document}


\begin{frontmatter}

\title{ On the Hermitian curve and its intersections with some conics} 
\runtitle{Hermitian curves, intersection with conics, affine-variety codes.}

{\author{Chiara Marcolla}} {\tt{(chiara.marcolla@unitn.it)}}\\
{Department of Mathematics, University of Trento, Italy}

{\author{Marco Pellegrini}} {\tt{(pellegrin@math.unifi.it)}}\\
{Department of Mathematics, University of Firenze, Italy}

{\author{Massimiliano Sala}} {\tt{(maxsalacodes@gmail.com)}}\\
{Department of Mathematics, University of Trento, Italy}

\runauthor{C.~Marcolla, M.~Pellegrini, M.~Sala}


\begin{abstract}
We classify completely the intersections of the Hermitian curve with
parabolas in the affine plane.
To obtain our results we employ well-known algebraic methods for
finite fields and geometric properties of the curve automorphisms. In
particular, we provide explicit counting formulas that have also
applications to some Hermitian codes. 
\end{abstract}
 

\begin{keyword}
Hermitian curve, intersection,  parabola.
\end{keyword}


\end{frontmatter}


\section{Introduction}
\label{intro}
\input{intro}


\section{Preliminary results}
\label{pre}
\input{pre}


\section{Intersection between Hermitian curve $\He$ and parabolas}
\label{parabole}
\input{parabole}


\section{Applications to coding theory}
\label{codHer}
\input{codHer}


\section{Conclusions and open problems}
\label{conc}
\input{conc}


\section*{Acknowledgements}
\label{ack}
\input{ack}


\bibliography{RefsCGC}

\end{document}

%% file: intro.tex
Let $q$ be a power of a prime. The \textit{Hermitian curve} $\He$ is the plane curve
defined over $\FQ$ by the affine equation $x^{q+1}=y^q+y$, where $x,y\in\FQ$.

This is the best-known example of maximal curve and there is a vast
literature on its properties, see \cite{CGC-cod-book-hirschfeld2008algebraic} for a recent survey.

Although a lot of research has been devoted to geometric properties of
$\He$, we present in this paper a  classification result,
providing for any $q$ the number of possible intersection points
between any parabola and $\He$.
With parabola we mean a curve $y=ax^2+bx+c$, where $x,y\in\FQ$ and $a, b, c$ are given with $a\not=0$ and $a, b, c \in\mathbb{F}_{q^2}$.
 Moreover, we can characterize precisely the parabolas obtaining a given number of intersection points and we can count them. More precisely,
given two curves $X$ and $Y$ lying in the affine plane $\mathbb{A}^2(\Fq)$ it is
interesting to know the number of (affine plane) points
that lie in both curves, disregarding multiplicity.
We call this number \textit{their planar intersection}.\\

\break

Our main result is the following
\begin{teo}[Theorem \ref{teo.principe} short version]
The only possible planar intersections of $\He$ and a parabola are:
\begin{itemize}
    \item[\ding{71}] for $q$ odd $\{0,1,q-1,q,$ $q+1,2q-1,2q\}$,
    \item[\ding{71}] for $q$ even $\{1,q-1,q+1,2q-1\}$.
\end{itemize}
For any possible planar intersection we have computed explicitly the exact number of parabolas sharing that value.
\end{teo}

The paper is organized as follows:
\begin{itemize}
\item In Section \ref{pre} we introduce the definitions of norm and trace functions and some lemmas that we need to prove our main theorem (Theorem~\ref{teo.principe}).  Finally, we sketch our
proving argument, that is, the use of the automorphism group for
$\He$.

\item In Section \ref{parabole}  we state and prove Theorem~\ref{teo.principe}.
The proof is divided in two main parts: in Subsection~\ref{odd} we deal with the odd-characteristics case and in Subsection~\ref{even} we deal with the even-characteristics case.

\item In Section \ref{codHer} we show the relation between our results and a family of Hermitian codes.

\item In Section \ref{conc} we draw some conclusions and propose some
open problems.

\end{itemize}

%% file: pre.tex
Let $\Fq$ be the finite field with $q$ elements, where $q$ is a power of a prime and let $\FQ$ be the finite field with $q^2$ elements. We call $\alpha$ a  primitive element of $\FQ$, and we consider $\beta=\alpha^{q+1}$ as a primitive element of $\Fq$. \\

The \textit{Hermitian curve} $\He$ is the plane curve defined over $\FQ$ by the affine equation $x^{q+1}=y^q+y$, where $x,y \,\in\,\FQ$. We recall that
this curve has genus $g = \frac{q(q-1)}{2}$ and has $n = q^3$ $\FQ$-rational affine
points and one point at infinity $P_{\infty}$, so it has $q^3+1$ rational points over $\FQ$ \cite{CGC-alg-art-rucsti94}.\\

We consider the \textit{norm} and the \textit{trace}, the two functions defined as follows. 
\begin{definition}\label{nt}
The \textbf{norm} $\mathrm{N}^{\mathbb{F}_{q^m}}_{\Fq}$ and the \textbf{trace}
$\mathrm{Tr}^{\mathbb{F}_{q^m}}_{\Fq}$ are two functions from $\mathbb{F}_{q^m}$ to
$\Fq$ such that
$$\N{\mathbb{F}_{q^m}}{\Fq}{x}=x^{1+q+\dots+q^{m-1}} \mbox{ and }
\Tr{\mathbb{F}_{q^m}}{\Fq}{x}=x+x^{q}+\dots+x^{q^{m-1}}.$$
\end{definition}%

We denote with N and Tr, respectively, the norm and the trace from $\FQ$ to $\Fq$.
It is clear that $\He=\{\No(x)=\Tra(y)\mid x,y\in\FQ\}$.\\

Using these functions, we define the map $F_a:\FQ\rightarrow\Fq$ such that
\begin{equation}\label{prin}
  F_a(x)=\No(x)-\Tra(ax^2).
\end{equation} 
We can note that the following property holds for the function $F_a$:

\begin{lemma}\label{lemmaFa}
If $\omega\in\Fq$, then $F_a(\omega x)=\omega^2 F_a(x)$.
\end{lemma}
\begin{proof}
Since $\omega\in\Fq$, we have $F_a(\omega x)=\No(\omega x)-\Tra(a(\omega x)^2)=
\omega^{q+1}x^{q+1}-a^q\omega^{2q}x^{2q}-a\omega^2 x^2=
\omega^2(x^{q+1}-a^q x^{2q}-ax^2)=\omega^2 F_a(x).$
\end{proof}

\subsection{Elementary results}

\begin{lemma}\label{sol}
Let $t\in\FQ^*$, then there is a solution of $x^{q-1}=t$ if and only if $\No(t)=1$.
In this case, $x^{q-1}=t$ has exactly $(q-1)$ distinct solutions in $\FQ$.
\end{lemma} 
\begin{proof}
This lemma is the Hilbert's Theorem 90 (see  Theorem 6.1 of \cite{CGC-alg-book-lang02}).
\end{proof}
\begin{remark}
We note that $4\No(a)=\No(2a)$ for any $a\in\FQ$.
\end{remark}

\begin{lemma}\label{aquad}
If $q$ is odd and $4\No(a)=1$ then $a$ is a square in $\FQ$.
\end{lemma} 
\begin{proof}
Let $a=\alpha^k$, so $4\No(a)=1$ implies that $4\alpha^{k(q+1)}=1$ i.e. $4\beta^k=1$.
Since $4$ is a square in $\Fq$, we have $4=\beta^{2t}$, that is, $\beta^{2t+k}=1$ and so $2t+k\equiv 0\mod q-1$.
Hence that $k$ is even and $a$ is a square.
\end{proof}




\begin{lemma}\label{corFON}
Let $f:\FQ\rightarrow\FQ$ such that $f(x)=2ax-x^q$. Then the equation $f(x)=k$
has $q$ distinct solutions if $k\in\I(f)$, otherwise it has $0$ solutions.
\begin{proof}
Since $f$ is $\Fq$-linear, our claim follows from standard results in linear algebra.
\end{proof}
\end{lemma}

\begin{lemma}\label{lemtraccia}
Let $y=ax^2+bx+\bar c$ and $y=ax^2+bx+c$ be two parabolas. If $\Tra(\bar c)=\Tra(c)$,
then the planar intersections between the Hermitian curve $\He$ and the parabolas are the same.
\end{lemma}
\begin{proof}
From a set $\{y=ax^2+bx+\bar c\}\cap\He$ and another set $\{y=ax^2+bx+c\}\cap\He$ we obtain by direct substitution respectively
$x^{q+1}=a^qx^{2q}+ax^2+b^qx^q+bx+\Tra(\bar c)$ and $x^{q+1}=a^qx^{2q}+ax^2+b^qx^q+bx+\Tra(c)$.
If $\Tra(\bar c)=\Tra(c)$, the two equations are identical.
\end{proof}

Finally, we recall the definition of \textit{quadratic character} of $\FF_q$.\\
Let $q$ be odd, then

$$
\eta (a) =\left\{\begin{array}{rl}
    1 & \mbox{ if } a \equiv  k^2 \mod q \\
   -1 & \mbox{ if } a \not\equiv k^2 \mod q \\
    0 & \mbox{ if } a\equiv 0 \mod q. \\
\end{array}\right.
$$

\begin{theorem}\label{teo.quad}
Let $f(x)=ax^2+bx+c \in \Fq[x]$ with $q$ odd and $a\ne 0$.\\ 
Let $d=b^2-4ac$, then
$$
\sum_{\gamma\in\Fq} \eta(f(\gamma)) =  
\left\{\begin{array}{cl}
   -\eta(a) & \mbox{ if } d \ne 0 \\
   (q-1)\eta(a)  & \mbox{ if } d=0. 
\end{array} \right.
$$
\begin{proof}
See Theorem 5.48 of \cite{CGC-cd-book-niederreiter97}
\end{proof}
\end{theorem}


\subsection{Automorphisms of Hermitian curve}\label{aut.her}

We consider an automorphism group $Aut(\He / \FQ)$ of the Hermitian curve over $\FQ$. $Aut(\He / \FQ)$ contains a subgroup $\Gamma$, such that any $\sigma\in\Gamma$ has the following form, as in \cite{CGC-alg-art-xing95} and in Section 8.2 of \cite{CGC-cd-book-stich}:
$$\sigma\left(\begin{array}{c}x \\ y\end{array}\right)=
\left(\begin{array}{c}\epsilon x+\gamma \\ \epsilon^{q+1}y+\epsilon\gamma^qx+\delta\end{array}\right)$$
with $(\gamma,\delta)\in\He$, $\epsilon\in\FQ^*$. 
Note that $\Gamma$ is also a subset of the group of affine
transformations preserving the set of $\FQ$-rational affine points of $\He$.\\
If we choose $\epsilon=1$ we obtain the following automorphisms 

\begin{equation}\label{autom}
  \left\{\begin{array}{l}
  x\longmapsto x+\gamma\\
  y\longmapsto y+\gamma^qx+\delta
\end{array}\right.\qquad\mbox{ with }(\gamma,\delta)\in\He,
\end{equation}
that form a subgroup $\Lambda$ with $q^3$ elements, see Section \texttt{II} of \cite{CGC-cod-art-stichtenoth1988note}.\\
The reason why we are interested in the curve automorphisms is the following.
If we apply any $\sigma$ to any curve $\xx$ in the affine plane, then the planar intersections  between
$\sigma(\xx)$ and $\He$ will be the same as the planar intersections between $\xx$ and $\He$. We recall that the number of planar intersection between two curves $X$ and $Y$ lying in the affine plane $\mathbb{A}^2(\Fq)$  is
 the number of (affine plane) points
that lie in both curves, disregarding multiplicity.
So, if we find out the number of intersections between $\xx$ and $\He$, we will
automatically have the number of intersection between $\sigma(\xx)$ and $\He$ for all $\sigma\in \Gamma$.
This is convenient because we can isolate special classes of parabolas
that act as representatives in the orbit $\{\sigma(\xx)\}_{\sigma\in\Gamma}$.
These special types of parabolas may be easier to handle.\\

Note that 
\noindent if we apply (\ref{autom}) to $y=ax^2$, we obtain
\begin{eqnarray}\label{par}
  y=ax^2+x(2a\gamma-\gamma^q)+a\gamma^2-\delta,
\end{eqnarray}
 
\noindent while if we apply (\ref{autom}) to $y=ax^2+c$ we obtain
\begin{eqnarray}\label{parc}
  y=ax^2+x(2a\gamma-\gamma^q)+a\gamma^2-\delta+c.
\end{eqnarray}

\noindent  In the general case, if we have $y=ax^2+bx+c$ and apply the automorphism (\ref{autom}) we obtain 
\begin{equation}\label{autpartot}
  y=ax^2+(2a\gamma-\gamma^q+b)x+a\gamma^2+b\gamma-\delta+c.
\end{equation}

To prove Theorem~\ref{teo.principe}, we have to study two distinct cases depending on the field characteristic.
Subsection~\ref{odd} is devoted to the proof of Theorem \ref{teo.principe} when the characteristic is odd,
while Subsection~\ref{even} is devoted to the proof of Theorem \ref{teo.principe} when the characteristic is even.

%% file: parabole.tex
We recall the definition of the Hermitian curve $\He$ on $\FQ$, i.e.
$$x^{q+1}=y^q+y, \mbox{ where } x,y\in\FQ.$$ 

\noindent The number of planar intersection between two curves $X$ and $Y$ lying in the affine plane $\mathbb{A}^2(\Fq)$  may have applications for the codes constructed from  $X$ and $Y$.
Regarding $\He$, it is interesting for coding theory applications
\cite{CGC-cod-art-couvreur2011dual},
\cite{CGC-cod-art-ballico2012goppa,CGC-cod-art-ballico2012geometry,CGC-cod-art-fontanari2011geometry}
to consider an arbitrary parabola $y=ax^2+bx+c$ over
$\FQ$ and to compute their planar intersection.
Moreover, it is essential to know precisely the number of parabolas
having a given planar intersection with $\He$.
Although nice partial results have been recently obtained in
\cite{CGC-alg-art-dondur10,CGC-alg-art-dondurkor09} (where a much more general
situation is treated), 
we present here for the first time a complete classification in the following theorem.\\

\begin{theorem}\label{teo.principe}
For $q$ odd, the only possible planar intersections of $\He$ and a parabola are $\{0,1,q-1,q,$ $q+1,2q-1,2q\}$.
For any possible planar intersection we provide in the next tables the exact number of parabolas sharing that value.

{\small{\begin{table}[h]
\centering
\begin{tabular}{|l||c|c|c|}
\hline
$\#\He\cap$ parabola & 0 & 1 & $q-1$  \\
\hline
$\#$ parabolas & $q^2(q+1)\frac{(q-1)}{2}$ & $q^2(q+1)\frac{q(q-3)}{2}$ & $q^2(q+1)\frac{q(q-1)^2}{2}$  \\
\hline
\multicolumn{4}{c}{$ \quad$} \\
\hline
$\# \He\cap$ parabola & $q$ &\multicolumn{2}{c|}{ $q+1$} \\
\hline
$\#$ parabolas & $q^2(q+1)(q^2-q+1)$ &\multicolumn{2}{c|}{$q^2(q+1)\frac{q(q-1)(q-3)}{2}$}  \\
\hline
\multicolumn{4}{c}{$ \quad$} \\
\hline
$\# \He\cap$ parabola & $2q-1$ & \multicolumn{2}{c|}{ $2q$}\\
\hline
$\#$ parabolas & $q^2(q+1)\frac{q(q-1)}{2}$ & \multicolumn{2}{c|}{$q^2(q+1)\frac{(q-1)}{2}$} \\
\hline
\end{tabular}
\end{table}}}

\break

For $q$ even, the only possible planar intersections  of $\He$ and a parabola  are $\{1,q-1,q+1,2q-1\}$.
For any possible planar intersection we provide in the next tables the exact number of parabolas sharing that value.

\begin{table}[h]
\centering
\begin{tabular}{|l||c|c|}
\hline
$\# \He\cap$ parabola &  1 & $q-1$  \\
\hline
$\#$ parabolas & $q^3(q+1)(\frac{q}{2}-1)$ & $q^3(q+1)(q-1)\frac{q}{2}$\\
\hline
\multicolumn{3}{c}{$ \quad$} \\
\hline
$\# \He\cap$ parabola &  $q+1$ & $2q-1$ \\
\hline
$\#$ parabolas & $q^3(q+1)(q-1)(\frac{q}{2}-1)$ & $q^3(q+1)\frac{q}{2}$ \\
\hline
\end{tabular}
\end{table} 

\end{theorem}

\subsection{Odd characteristics}\label{odd}

In this subsection, $q$ is always odd.\\


We provide now the sketch of the proof. Let $a,b,c\in\FQ$ and $a\ne 0$. We organize the proof in two main parts: in the first we analyze the intersection between $\He$ and parabola of type $y=ax^2+ c$ 
(Subsection \ref{sub_ac}) and in the second one we compute the number of intersection between $\He$ and the more general case $y=ax^2+bx+ c$
(Subsection~\ref{sub_abc}) for nearly all cases. The last cases are dealt with a simple counting argument. For each of these subsections the proof is organized in several steps that are summarized in the following scheme.
\begin{itemize}
    \item[1)] $\mathbf{\He \cap \{y=ax^2+ c\}}.$\\
By intersecting $\He$ with $\{y=ax^2+ c\}$ we are led to consider the equation
$F_{a}(x)=-\Tra(c)$, where $F_{a}(x)$ is as in (\ref{prin}), that is, 
$$
\begin{array}{l}
F_{a}(x)= a^qx^2\left(x^{q-1}-\frac{1+\sqrt{\Delta}}{2a^q}\right)
\left(x^{q-1}-\frac{1-\sqrt{\Delta}}{2a^q}\right) \mbox{ with }\\
\Delta =  1-4\No(a).
\end{array}
$$
We call $\mathcal{A}:= \He \cap \{y=ax^2+ c\}$ with $a,c$ fixed. We are interested in the $x$ position of points in $\mathcal A$ (that we call ``solutions of $\mathcal A$''). That is, the number of the $x$'s that verify the equation $F_{a}(x)=-\Tra(c)$, with $a,c$ fixed.\\ 
We consider two subcases:
\begin{itemize}
    \item[$*$] $\Tra(c)=0$. If
\begin{itemize}
    \item[-]  $\Delta=0 \implies $ $\mathcal{A}$ has $q$ solutions.
    \item[-]  $\Delta=1 \implies $ $\mathcal{A}$ has no solution.
    \item[-]  $\Delta=z^2 \in \Fq \backslash \{0,1\}\implies 
\left\{\begin{array}{lcl}
z\in\Fq & \implies & \mathcal{A} \mbox{ has } 1 \mbox{ solution.}\\
z\not\in\Fq & \implies & \mathcal{A} \mbox{ has } 2q-1 \mbox{ solutions.}\\
\end{array}\right.$\\
\end{itemize}
Now we apply the \textit{automorphism} (\ref{autom}) to the parabola $y=ax^2+ c$ and we obtain $y=ax^2+ b'x+ c'$, for some $b',c'\in\FQ$. We prove that when $\Delta\ne 0$, we have exactly $q^3$ distinct parabolas that share with $y=ax^2$ the planar intersection.\\

    \item[$*$]  $\Tra(c)\ne 0$. If
\begin{itemize}
    \item[-]  $\Delta=0 \implies  
\left\{\begin{array}{l}
 \mathcal{A} \mbox{ has } 0 \mbox{ solution.}\\
 \mathcal{A} \mbox{ has } 2q \mbox{ solutions.}\\
\end{array}\right. \mbox{Both cases depend on } a.$
    \item[-]  $\Delta=z^2 \in \Fq \backslash \{0,1\}\implies 
\left\{\begin{array}{lcl}
z\in\Fq & \implies & \mathcal{A} \mbox{ has } q+1 \mbox{ solutions.}\\
z\not\in\Fq & \implies & \mathcal{A} \mbox{ has } q-1 \mbox{ solutions.}\\
\end{array}\right.$
\end{itemize}
\end{itemize}
    \item[2)]  $\mathbf{\He \cap \{y=ax^2+ bx+ c\}}$.\\
We apply the \textit{automorphism} (\ref{autom}) to the parabola $y=ax^2+ bx+ c$ and we obtain $y=ax^2+ b'x+ c'$, for some $b',c'\in\FQ$.
We consider two subcases:
\begin{itemize}
    \item[$*$] $b'\ne 0$.  If $4\No(a)=1$, that is, $\Delta=0$ we can write any such parabola as $y = a(x+v)^2$, where $v\in\FQ$ such that $v^q +2av\ne 0$. So we have $q$ points of intersection between $\He$ and $y = a(x+v)^2$, with $a,v$ fixed.

    \item[$*$]  $b'= 0$. Therefore, we apparent fall in the case $1$). However, we actually dealing with different conditions of type $\Tra(c)=0,\Tra(c)\ne 0,$ $\Delta =0$ and $\Delta \ne 0$.\\
So we apply the \textit{automorphism} (\ref{autom}) to the parabola $y=ax^2+ c'$ and we obtain $y=ax^2+  bx+ \bar c$, for some $\bar c\in\FQ$. We want to understand \textit{how many different parabolas} we can obtain when $\Tra(\bar c)\ne 0$. We divide two cases:
\begin{itemize}
    \item[-]  $\Delta=0\implies $ $q^2(q+1)$ possible parabolas (fixing $a$).
    \item[-]  $\Delta=z^2 \in \Fq \backslash \{0,1\}\implies $ $q^3(q-1)$ possible parabolas (fixing $a$).
\end{itemize}
\end{itemize}
\item[3)] Finally, we obtain the number of parabolas that have $q$ intersections with the Hermitian curve by a simple counting argument. 
\end{itemize}


\subsubsection{Intersection between $\He$ and $y=ax^2+ c$}\label{sub_ac}
Intersecting a parabola of the form $y=ax^2+ c$ with
the Hermitian curve, we obtain $x^{q+1}=a^q x^{2q}+ax^2+\Tra(c)$ which is equivalent to
\begin{equation}\label{eqiniziale}
  \No(x)-\Tra(ax^2)=F_a(x)=\Tra(c).
\end{equation}
We have to study the number of solutions of (\ref{eqiniziale}).
From this equation we get $a^q x^{2q}-x^{q+1}+ax^2=-\Tra(c)$, that is,
\begin{equation}\label{eqquad}
x^2(a^q x^{2q-2}-x^{q-1}+a)=-\Tra(c).
\end{equation}
Now we set $x^{q-1}=t$ and we factorize the polynomial $a^q t^2-t+a$ in $\FQ[t]$, obtaining
$$t_{1,2}=\frac{1\pm\sqrt{1-4\No(a)}}{2a^q}=\frac{1\pm\sqrt{\Delta}}{2a^q}$$
where  $\Delta=1-4\No(a)$. So equation (\ref{eqquad}) becomes
\begin{equation}\label{eqini}
a^qx^2\left(x^{q-1}-\frac{1+\sqrt{\Delta}}{2a^q}\right)
\left(x^{q-1}-\frac{1-\sqrt{\Delta}}{2a^q}\right)=-\Tra(c).    
\end{equation}
Since $\Delta\in\Fq$, there exists $z\in\FQ$ such that $\Delta=z^2$, and so  equation (\ref{eqini}) is in $\FQ[x]$.\\ 
Note that 
$$\Delta=0\iff \No(2a)=1.$$ 
So, in this special case, (\ref{eqini}) becomes 
$a^qx^2(x^{q-1}-2a)^2=-\Tra(c)$. We have proved the following lemma.
\begin{lemma}\label{deltaquad}
By intersecting a parabola $y=ax^2+c$, where $\No(2a)=1$,  and the Hermitian curve, we obtain the following equation
$$a^qx^2(x^{q-1}-2a)^2=-\Tra(c).$$
\end{lemma}
\noindent Recall that $\alpha$ is a primitive element of $\FQ$ and $\beta=\alpha^{q+1} $ is a primitive element of~$\Fq$.\\

\begin{lemma}\label{lem:ab}
Let $x=\alpha^j\beta^i$, with $j=0,\ldots,q$ and $i=0,\ldots,q-2$. Then
\begin{itemize}
  \item[\ding{71}] If $4\No(a)\ne 1$, then the non-zero values $F_a(\alpha^j\beta^i)$ are all the elements of~$\Fq^*$.
  \item[\ding{71}] If $4\No(a)=1$, then the non-zero values $F_a(\alpha^j \beta^i)$ are half of the elements of~$\Fq^*$.
\end{itemize}
\end{lemma}
\begin{proof}
We recall that $F_a(x)=x^{q+1}-a^q x^{2q}-ax^2$.
We fix an index $j$ such that $F_a(\alpha^j)\ne 0$.
The set of the values 
$$
\{F_a(\alpha^j \beta^i)\}_{0\leq i\leq q-2}=
\{\beta^{2i}F_a(\alpha^j)\}_{0\leq i \leq q-2}
$$ 
contains half of the elements of $\Fq^*$, since $q$ is odd (and $\frac{q-1}{2}$ is an integer) and so $\beta^{2(\frac{q-1}{2})}=\alpha^{q^2-1}=1$. In particular, if $F_a(\alpha^j)$ is a square then $\{\beta^{2i}F_a(\alpha^j)\}_{0\leq i \leq q-2}$ are all squares of $\Fq^*$ (and vice-versa if it is a non-square).\\
Suppose that $F_a(1)=1-a^q-a$ is a square and let $\bar x=y+\gamma \in\FQ$, where $y\in\Fq$, $\gamma \ne 0$ and $\Tra(\gamma)=0$. Then
$$
\begin{array}{ll}
 F_a(\bar x)&=-a^q(y-\gamma)^2+(y-\gamma)(y+\gamma)-a(y+\gamma)^2\\
&=y^2(1-a^q-a)+2\gamma y(a^q-a)-\gamma^2(a^q+a+1) :=f_{\gamma}(y).  
\end{array}
$$
Since $f_{\gamma}(y)\in\Fq[y]$, we can apply Theorem \ref{teo.quad}, where $d=4\gamma^2(1-4\No(a))=4\gamma^2\Delta$.
\begin{itemize}
\item[\ding{71}] If $\Delta\ne 0$ we have
$\sum_{\epsilon\in\Fq} \eta(f_{\gamma}(\epsilon)) = -\eta(1-a^q-a) = -1$, that is, there exists at least a $y_1\in\Fq$ such that $f_{\gamma}(y_1)$ is not a square. Let  $\bar x_1=y_1+\gamma \in\FQ$, then $F_a( \bar x_1 \beta^i)$ and $F_a(1 \beta^i)$, varying $i=0,\ldots,q-2$, are all elements of $\Fq^*$ (that are all non-squares and all squares respectively).
\item[\ding{71}]  If $\Delta =0$, by Lemma \ref{deltaquad}, $F_a(x)$ becomes $-a^qx^2(x^{q-1}-2a)^2$,
so $\beta^{2i}F_a(\alpha^j)=-a^q\beta^{2i}(\alpha^{jq}-2a\alpha^j)^2$
and they are half of the elements of $\Fq^*$. In particular if $-a^q$ is a square we obtain all squares of $\Fq^*$, vice-versa, if $\eta(-a^q) = -1$, we have all non-squares of $\Fq^*$.
\end{itemize} 
\end{proof}

Now we study the number of solutions of equation (\ref{eqiniziale}),
analyzing two cases: when $\Tra(c)=0$ and when $\Tra(c)\ne 0$.\\

\begin{itemize}
  \item[$*$] Case $\Tra(c)=0$. By Lemma \ref{lemtraccia}, it is enough to study
the case $c=0$, which is the intersection between $\He$ and $y=ax^2$.
By (\ref{eqini}) we have
$$a^qx^2\left(x^{q-1}-\frac{1+\sqrt{\Delta}}{2a^q}\right)
\left(x^{q-1}-\frac{1-\sqrt{\Delta}}{2a^q}\right)=0.$$
We must differentiate our argument depending on $\Delta$. Recall that $\Delta\in\Fq$.
\begin{itemize}
    \item[-] $\Delta=0$.
By Lemma \ref{deltaquad}, (\ref{eqini}) becomes
$$
a^qx^2(x^{q-1}-2a)^2=0.
$$
So we have always one solution $x=0$ and the solutions of $x^{q-1}=2a$.
Since $\No(2a)=1$, by Lemma \ref{sol}, the number of solutions of $x^{q-1}=2a$ are $q-1$.
Therefore, in this case, we have $q$ points of intersections between the parabola and the Hermitian curve $\He$. 

\noindent By condition on $a$, i.e. $\No(2a)=1$, we have $(q+1)$ distinct $a$'s.
    \item[-]$\Delta=1$. That is, $\No(2a)=0\iff a=0$, which is impossible.
    \item[-]$\Delta\in\Fq\backslash\{0,1\}$.
We note that any element in $\Fq$ can always be written as $z^2$ with $z\in\FQ$. So let $\Delta=z^2$.
In order to study the solutions of (\ref{eqini}), we can consider the solutions of the following equations
\begin{equation}\label{stella}
    x^{q-1}=\frac{1\pm z}{2a^q}.
\end{equation}
By Lemma \ref{sol} we know that $x^{q-1}=\frac{1+z}{2a^q}$ has solutions
if and only if $\No\left(\frac{1+z}{2a^q}\right)=1$. Note that
{\small{$$\No\left(\frac{1+z}{2a^q}\right)=1\iff\frac{(1+z)^{q+1}}{1-z^2}=1\iff 1-z=(1+z)^q\iff -z=z^q$$}}

\noindent We obtain the same result for $x^{q-1}=\frac{1-z}{2a^q}$.\\
If (\ref{stella}) has a solution $x$ and $z\in\Fq$, then $z$ simultaneously satisfies $z^q=z$ and $z^q=-z$. Since $q$ is odd, this is possible only when $z=0$, which implies $\Delta=0$, which is not admissible.\\

\noindent  Returning to count the intersection points, thanks to the previous discussion of the solution of (\ref{stella}), we have to consider two distinct cases:
\begin{itemize}
    \item[\ding{71}] $z=z^q$, that is,  $z\in\Fq$.  Since $z\ne 0,1$, there
    are $\frac{q-1}{2}-1=\frac{q-3}{2}$ possible values of $z^2$, and so we have $(q+1)\frac{q-3}{2}$ values of $a$. 
In this case, the parabola $y=ax^2+c$ intersects $\He$ in only $one$ point (with $x=0$).
    \item[\ding{71}] $z=-z^q$. The equation $-z=z^q$ has only one solution in $\Fq$, so the other $q-1$ solutions are in $\FQ\backslash\Fq$.
For such $z$, we have $2(q-1)+1=2q-1$ points of intersection.
That is, $q-1$ solutions from equation $x^{q-1}=\frac{1-z}{2a^q}$, $q-1$ solutions from equation $x^{q-1}=\frac{1+z}{2a^q}$ and one point from $x=0$.\\
It is simple to verify that the number of $z^2$ such that $z\in\FQ\backslash\Fq$ is $\frac{q-1}{2}$.
So we have $(q+1)\frac{q-1}{2}$ values of $a$ for which we have exactly
$2q-1$ points of intersection between $y=ax^2+c$ and $\He$.
\end{itemize}
\end{itemize}

Now we apply the automorphism (\ref{autom}) and we want to compute how many different parabolas
we can obtain. Applying (\ref{autom}) to $y=ax^2$ we obtain (\ref{par}):
$$
  y=ax^2+x(2a\gamma-\gamma^q)+a\gamma^2-\delta.
$$
For the moment, we restrict our counting argument to the case $\Delta\ne 0$.
We note that if $\Delta\ne 0$, we have a maximal orbit, that is,
all possible parabolas are distinct (there are $q^3$ because $\Gamma$ has $q^3$ elements).
In other words, we claim that it is impossible that we obtain
two equal parabolas with $(\gamma,\delta)\ne(\bar\gamma,\bar\delta)$.
To prove that, we have to solve the following system:

{\footnotesize{$$
\left\{\begin{array}{l}
2a\bar\gamma-\bar\gamma^q=2a\gamma-\gamma^q\\
a\bar\gamma^2-\bar\delta=a\gamma^2-\delta\\
\gamma^{q+1}=\delta^q+\delta\\
\bar\gamma^{q+1}=\bar\delta^q+\bar\delta\\
1-4a^{q+1}\ne 0.
\end{array}\right.
$$}}

\noindent However, $2a\bar\gamma-\bar\gamma^q=2a\gamma-\gamma^q\iff
2a(\bar\gamma-\gamma)=\bar\gamma^q-\gamma^q=(\bar\gamma-\gamma)^q$. Raising the equation to the power of $q+1$, we obtain  $(2a)^{q+1}(\bar\gamma-\gamma)^{q+1}=(\bar\gamma-\gamma)^{q^2+q}$, that is, $4a^{q+1}(\bar\gamma-\gamma)^q(\bar\gamma-\gamma)=
(\bar\gamma-\gamma)(\bar\gamma-\gamma)^q$, which is equivalent to $4a^{q+1}=1$. This is impossible, since $\Delta\ne 0$.\\
Hence, when $\Delta\ne 0$, we have exactly $q^3$ distinct parabolas that have
the same planar intersections  with $\He$ as $y=ax^2$ has.

  \item[$*$] Case $y=ax^2+c$, with $\Tra(c)\ne 0$.
As in previous case, we have to differentiate depending on $\Delta$.
\begin{itemize}
    \item[-] If $\Delta=z^2$ and $z\in\Fq$, we know that $F_a(x)$ vanishes only if $x=0$.
If $x\ne 0$, then by Lemma \ref{lem:ab}, $F_a(\beta^i\alpha^j)=\beta^{2i}F_a(\alpha^j)=t$
assumes every value of $\Fq^*$. 
But $x =\beta^i\alpha^j$ assumes $q^2- 1$ distinct values, varying $i$ and $j$.
So every $t$ is obtained $q + 1$ times ($F_a(x)$ is a polynomial of degree $q + 1$). 
Hence, the equation $F_a(x)=\Tra(c)$ has exactly $q+1$ solutions.
    \item[-] If $\Delta=z^2$ and $z\in\FQ\backslash\Fq$, we know that $F_a(x)=0$
has $2q-1$ solutions. So there are exactly two distinct values of $j$ such that
$F_a(\alpha^j)=0$, one for each equation $x^{q-1}=\frac{1\pm z}{2a^q}$ (to find the $q-1$ solutions, we vary $i$).
So every value in $\Fq^*$ is obtained $q-1$ times. Hence,
the equation $F_a(x)=\Tra(c)$ has exactly $q-1$ solutions.
    \item[-] If $\Delta=0$ we have $4a^{q+1}=1$.\\ 
So (\ref{prin}) can be written as $ a^qx^2(x^{q-1}-2a)^2=-\Tra(c)$, that is,
\hspace{-1.5cm}\begin{equation}\label{dzero}
  x^2(x^{q-1}-2a)^2=-4a\Tra(c)=-4a\beta^r
\end{equation}
for some fixed $r$ with $1\leq r\leq q-1$.\\ 
Note that (\ref{dzero}) can be written as $f(x)^2=-4a\Tra(c)$,
where $f$ is as in Lemma \ref{corFON}, that is, $f(x)=x^{q}-2ax$.\\
We note that $-4a\beta^r$ is always a square in $\FQ$.
In fact $-4\beta^r$ is a square because it lies in $\Fq$,
and also $a$ is a square by Lemma~\ref{aquad}. Let us write $-4a\beta^r=\alpha^{2h}$, so (\ref{dzero}) becomes
$x(x^{q-1}-2a)=\pm\,\alpha^{\,h}$ where $0\leq h\leq \frac{q^2-1}{2}$.

\noindent We consider the ``positive'' case:
\begin{equation}\label{cond1}
f(x)=x^{q}-2ax=\alpha^{\,h}.
\end{equation} 
\noindent It is simple to prove that if $x$ is a solution of equation
(\ref{cond1}), then $-x$ is a solution of the equation $x^q-2ax=-\alpha^h$.
So by  Lemma~\ref{corFON} the equation $F_a(x)=\Tra(c)$ has $0$ solutions
if $\alpha^h$ is not in $\I(f)$ or $2q$ solution if $\alpha^h$ is in $\I(f)$.

\end{itemize}
\end{itemize}

\subsubsection{Intersection between $\He$ and $y=ax^2+bx+ c$}\label{sub_abc}
We consider a parabola $y=ax^2+bx+c$, apply the automorphism (\ref{autom})
and we obtain (\ref{autpartot}).\\
Note that, for any $k\in\FQ$,
\begin{equation}\label{bdiv}
  2a\gamma-\gamma^q+b=k\implies 2ab^q+b=2ak^q+k,
\end{equation}
because $b^q=(k-2a\gamma+\gamma^q)^q=k^q-\frac{1}{2a}\gamma^q+\gamma=k^q+
\frac{1}{2a}(-\gamma^q+2a\gamma)=k^q+\frac{1}{2a}(k-b)$.\\

A consequence is that $2a\gamma-\gamma^q+b=0\implies 2ab^q+b=0$.\\

We consider two distinct cases $2a\gamma-\gamma^q+b=0$ and $2a\gamma-\gamma^q+b\ne 0$.\\


\noindent\begin{tabular}{|c|}
\hline
$2a\gamma-\gamma^q+b\ne 0$.\\
\hline
\end{tabular}\\

\begin{thm}
Let $y=ax^2+bx+c$ be a parabola with $2ab^q+b\ne 0$ and $\No(2a)=~1$.
Then there exists $\gamma$ such that for any $\delta$, applying the
automorphism~(\ref{autom}), we obtain $y=ax^2+(2a\gamma-\gamma^q+b)x+a\gamma^2+ b\gamma-\delta+c$,
with $2a\gamma-\gamma^q+b\ne 0$. We can write any such parabola as $y=(ux+uv)^2$ where $a=u^2$ and $v^q+2av\ne 0$.
\end{thm}  
\begin{proof}
Because of (\ref{bdiv}) with $k\ne 0$ we have that, since $2ab^q+b\ne 0$,
there exists $\gamma$ such that $2a\gamma-\gamma^q+b\ne 0$.\\
Let $k\in \FQ\backslash\{0\}$ such that $2a\gamma-\gamma^q+ b=k\ne 0$. By  Lemma
\ref{corFON}, if there exists at least one solution of $2a\gamma-\gamma^q=k-b$,
then there exist $q$ solutions. So we have at least $q$ different $\gamma$'s
that verify the previous equation.\\
To prove that any parabola as in (\ref{autpartot}) can be written $y=(ux+uv)^2$
with $a=u^2$ and $v^q+2av\ne 0$, we claim that it is sufficient to prove that the solutions of the following system contain all $c$'s.
\begin{equation}\label{c}
\left\{\begin{array}{l}
2a\gamma-\gamma^q+b=2av\ne 0\\
a\gamma^2+b\gamma-\delta+c=av^2\ne 0\\
\gamma^{q+1}=\delta^q+\delta\\
1-4a^{q+1}=0
\end{array}\right.
\end{equation}
In fact, system \eqref{c} is obtain as follows. We have the first two equations comparing $y=(ux+uv)^2$ and $y=ax^2+(2a\gamma-\gamma^q+b)x+a\gamma^2+ b\gamma-\delta+c$. 
Note that by Lemma \ref{aquad}, $a=u^2$ is a square so $y=(ux+uv)^2=a(x+v)^2$. 
The third equation denote that the point $(\gamma,\delta)\in\He$ and finally, the last equation is a necessary condition to verify this theorem, that is, $\No(2a)=1$.\\

\noindent We are ready to show the desired property of \eqref{c}'s solutions. Using (\ref{bdiv}), we first observe that the first equation of \eqref{c} 
implies that $v^q+2av\ne 0$. Indeed if we consider (\ref{bdiv}) with $k=2av$,
we have $0\ne 2ab^q+b=2ak^q+k=2a(2av)^q+2av=v^q+2av$.\\
Now, we prove that  \eqref{c} has $q^2$ different $c$'s in its solutions, that is, all possible $c$'s.
From the automorphism point of view, it is enough to prove that the point $(\gamma,\delta)$ of $\He$ are in bijection with the $c$'s contained in solutions of \eqref{c}. Which means that two distinct automorphisms (of the considered type) sends the curve in two distinct curves.\\
Multiplying the first equation by $\gamma$  we substitute $2a\gamma^2$ in the second equation multiplied by 2 and, using the curve equation $\gamma^{q+1}=\delta^q+\delta$, we obtain
\begin{equation}\label{eq.c}
2c=2av^2+\delta-\delta^q-\gamma(2av+b).
\end{equation}
Suppose by contradiction that there exist two points  $(\gamma_1,\delta_1),(\gamma_2,\delta_2)\in\He$ such that $(\gamma_1,\delta_1)\ne(\gamma_2,\delta_2)$ but $c_1=c_2$, where
$c_i$ are as in (\ref{eq.c}). Since  $c_1=c_2$, we would obtain 
\begin{equation}\label{eq.c1}
\gamma_1(2av+b)-\delta_1+\delta_1^q=\gamma_2(2av+b)-\delta_2+\delta_2^q.
\end{equation}
On the other hand, if we raise (\ref{eq.c1}) to the power of $q$ and substitute $\gamma_i^q$ with the first equation of (\ref{c}), that is,
$\gamma_i^q=2a\gamma_i+b-2av$ for $i=1,2$,  we obtain
\begin{equation}\label{eq.c2}
2a\gamma_1(2av+b)^q-\delta_1^q+\delta_1=2a\gamma_2(2av+b)-\delta_2^q+\delta_2.
\end{equation}
Summing the equations (\ref{eq.c1}) and (\ref{eq.c2}) we obtain 
$\gamma_1[(2av+b)+2a(2av+b)^q]=\gamma_2[(2av+b)+2a(2av+b)^q]$, that is, $\gamma_1=\gamma_2$ if $(2av+b)+2a(2av+b)^q\ne 0$.  Note $(2av+b)+2a(2av+b)^q=2av+b+4a^{q+1}v^q+2ab^q= v^q+2av+ 2ab^q+b= 2(2ab^q+b)\ne 0$, where we use the last equation of (\ref{c}), so it must be $\gamma_1=\gamma_2$. However, if $\gamma_1=\gamma_2$ then $\delta_1=\delta_2$ by the second equation of (\ref{c}), and this implies a contradiction. Therefore, for any two different points  $(\gamma_1,\delta_1),(\gamma_2,\delta_2)\in\He$ we have that $c_1\ne c_2$.
To conclude, we show that we have $q^2$ different $c$'s. By the second equation we have $c=\delta+av^2-a\gamma^2-b\gamma$. So, for any $\gamma$ (and there are $q$ possible $\gamma$'s), there are $q$ distinct $\delta$'s (by the curve equation). 
\end{proof}

\begin{thm}
Let $a,v\in\FQ$ such that $\No(2a)=1$ and $v^q+2av\ne 0$. Then the Hermitian curve $\He$
intersects the parabola $y=a(x+v)^2$ in $q$ points.
\end{thm}
\begin{proof}
We have to solve the system
$$
\left\{\begin{array}{l}
y=(ux+uv)^2\\
x^{q+1}=y^q+y
\end{array}\right.\implies x^{q+1}=(ux+uv)^{2q}+(u+uv)^2
$$
By a change of variables $z=ux+uv$, we obtain $(\frac{z-uv}{u})^{q+1}=z^{2q}+z^2$, so we have
$$
-(uv)z^q-(uv)^{q}z+(uv)^{q+1}=u^{q+1}z^{2q}+u^{q+1}z^2-z^{q+1}=
u^{q+1}(z^q-2u^{q+1}z)^2.
$$
Since $\No(2a)=1$ and $a=u^2$, we have $u^{q+1}=\pm \frac{1}{2}$ and so
\begin{align}
  \frac{1}{2}(z^q-z)^2=\No(uv)-\Tra(z(uv)^q)\label{eqz}\\
  -\frac{1}{2}(z^q+z)^2=\No(uv)-\Tra(z(uv)^q)\label{eqzm}
\end{align}

\break

\noindent We consider two cases:
\begin{itemize}
  \item[\ding{71}] If $u^{q+1}=\frac{1}{2}$, we claim that if $z^q-z\ne 0$, then $(z^q-z)^2$ is not a square in~$\Fq$. 
In fact, suppose by contradiction that $(z^q-z)^2=\beta^{2r}$,
then $z^q-z=\beta^{r}\in \Fq$ but also $z^q+z\in\Fq$, so $-2z\in\Fq\iff z\in\Fq$
and so $z^q-z=0$, which is impossible.
  \item[\ding{71}] If $u^{q+1}=-\frac{1}{2}$, we can note that $(z^q+z)^2$ is a square in $\Fq$, because $z^q+z\in\Fq$.
\end{itemize}
Let $t=\No(uv)-\Tra(z(uv)^q)$. So $t\in \Fq$.
Due to (\ref{eqz}) we have $2t=(z^q-z)^2$, while (\ref{eqzm}) becomes $-2t=(z^q+z)^2$.\\

\noindent When $u^{q+1}=\frac{1}{2}$, we have $\frac{q-1}{2}$ values of $t$ (that are all the non-squares) and $t=0$, whereas when $u^{q+1}=-\frac{1}{2}$, we have $\frac{q-1}{2}$ values of $t$
(that are all the squares) and $t=0$.\\
\noindent  Now we consider separately the cases $t=0$ and $t\ne 0$.

\begin{itemize}
  \item[\ding{71}] We claim that if $t=0$ and $u^{q+1}=\pm \frac{1}{2}$ then $z\in\Fq$. Whereas if $t=0$ and $u^{q+1}=-\frac{1}{2}$ then $z\in\FQ\backslash\Fq$.
  We show only the case $u^{q+1}=\frac{1}{2}$. With these assumptions (\ref{eqz}) becomes
$$
-(uv)z^q-(uv)^{q}z+(uv)^{q+1}=0\iff z=\frac{(uv)^{q+1}}{(uv)^{q}+uv}.
$$
We can note that since $v^q+2av\ne 0$, then $(uv)^{q}+uv\ne 0$.
In fact, suppose that $v^q+2av=0$, then $(uv)^{q}+uv=-\frac{1}{2u}2av+uv=0$.\\
We have to verify that $z^q=z$. Indeed $z^q=\frac{(uv)^{q+1}}{uv+(uv)^{q}}=z$.\\
Similar computations (here omitted) show the case $u^{q+1}=-\frac{1}{2}$.
  \item[\ding{71}] We claim that if $t\ne 0\mbox{ and } u^{q+1}=\pm \frac{1}{2}\implies z\not\in\Fq$.
  With these assumptions, we show only the case $u^{q+1}=\frac{1}{2}$. We have $(z^q-z)^2=2t=\alpha^{2r}$,
  that is, $z^q=z \pm \alpha^r$. Now we substitute $z^q$ in
$-(uv)z^q-(uv)^{q}z+(uv)^{q+1}=t $ and we obtain
$-(uv)(\pm\alpha^r+z)-(uv)^{q}z+(uv)^{q+1}=\frac{1}{2}\alpha^{2r}$, that is,
\begin{equation}\label{solz}
z=\frac{(uv)^{q+1}-\frac{1}{2}\alpha^{2r}\mp uv\alpha^r}{\Tra(uv)}
\end{equation}
We can note that $\alpha^{qr}=-\alpha^r$, in fact $2t=\alpha^{2r}\in\Fq$,
so $(\alpha^{2r})^q=\alpha^{2r}$, that is, $\alpha^{rq}=\pm\alpha^{r}$
but $\alpha^r\not\in\Fq$ (since $2t$ is not a square in $\Fq$) so $\alpha^{qr}=-\alpha^r$. We have thus proved
$$
z=\frac{(uv)^{q+1}-\frac{1}{2}\alpha^{2r}\mp uv\alpha^r}{\Tra(uv)}\mbox{ and }
z^q=\frac{(uv)^{q+1}-\frac{1}{2}\alpha^{2r}\pm (uv)^q \alpha^r}{\Tra(uv)}
$$

\noindent Now we have to verify that the two $z$'s as in (\ref{solz}) are solutions of (\ref{eqz}).
We have $z^q-z=\pm\alpha^r$ and $\No(uv)-\Tra(z(uv)^q)=t$. So
$$\pm\alpha^r=z^q-z\iff\pm \Tra(uv)\alpha^r=\pm(uv)^{q}\alpha^{r}\pm uv\alpha^r$$
and
{\small{
$$
\begin{array}{cl}
 & (uv)^{q+1}-z(uv)^q-z^q(uv)=t\\
\iff & (uv)^{q+1}\Tra(uv)-(uv)^q((uv)^{q+1}-\frac{1}{2}\alpha^{2r}\mp uv\alpha^r)+\\
&-uv((uv)^{q+1}-\frac{1}{2}\alpha^{2r}\pm uv \alpha^r)= \Tra(uv) t\\
\iff & (uv)^{q+1}\Tra(uv)+t\, \Tra(uv)-(uv)^{2q+1}-(uv)^{q+2}= \Tra(uv) t\\
\iff & (uv)^{q+1}\Tra(uv)-(uv)^{2q+1}-(uv)^{q+2}= 0.\\
\end{array}
$$}}

\noindent So the $z$'s are solutions of (\ref{eqz}).\\
Similar computations (omitted here)  show the case $u^{q+1}=-\frac{1}{2}$.
\end{itemize}
Therefore, we have two solutions for any $t$ not a square in $\Fq^*$
and we have only one solution when $t=0$.
That is, we get a total of $2\cdot\frac{q-1}{2}+1=q$ intersections.

\noindent The same holds for the case with $u^{q+1}=-\frac{1}{2}$.
\end{proof}


Now we consider the second case.\\

\noindent\begin{tabular}{|c|}
\hline
$2a\gamma-\gamma^q+b=0$.\\
\hline
\end{tabular}
$\,$\\

We note that if $2a\gamma-\gamma^q+b=0$ then $2ab^q+b=0$, and so (\ref{autpartot}) is actually $y=ax^2+\bar c$, where $\bar c =a\gamma^2+b\gamma-\delta+c \in \FQ$. With abuse of notation, we will write  $\bar c = c$, that is, $y=ax^2+ c$. Now we apply the automorphism (\ref{autom})
to the parabola $y=ax^2+c$ and we obtain (\ref{parc}).\\

 We study two different cases: if

\begin{itemize}
  \item[$*$] $\Delta\ne 0$, the parabolas in (\ref{parc}) are all distinct.

The number of values of $c$ such that $\Tra(c)\ne 0$ are exactly $q^2-q$,
but we must be careful and not count twice the same parabola.
In particular, if two parabolas share $a$ and $b$, then they are in the
same orbit if $\Tra(c)=\Tra(c')$.
So we must consider only one of these for any non-zero $\Tra(c)$. Thus there are  $q-1$ values. \\

\noindent Summarizing:
\begin{itemize}
    \item[-] If $\Delta=z^2$ and $z\in\Fq$ (and $\Tra(c)\ne 0$), then the number of parabolas with $q+1$ intersections is 
$$
\underbrace{(q+1)\frac{q-3}{2}}_{a}\underbrace{q^3(q-1)}_{b,c}=\frac{1}{2}
q^3(q^2-1)(q-3).
$$
    \item[-] If $\Delta=z^2$ and $z^q+z=0$ (and $\Tra(c)\ne 0$), then the number of parabolas with $q-1$ intersections is 
$$
\underbrace{(q+1)\frac{q-1}{2}}_{a}\underbrace{q^3(q-1)}_{b,c}=
\frac{1}{2}q^3(q+1)(q-1)^2.
$$
\end{itemize}
  \item[$*$] $\Delta=0$, that is, $4a^{q+1}=1$,
we want to understand how many different parabolas of the type $y=ax^2+bx+\bar c$
(with $a$ fixed) we can obtain. So we have to study the number of pairs $(b,\bar c)$.

We note that  
\begin{equation}\label{tracce}
  \Tra(\bar c)=a^qb^2+\Tra(c).
\end{equation}
In fact
$$
\begin{array}{lll}
\Tra(\bar c) & = & (a\gamma^2-\delta)^q+a\gamma^2-\delta+\Tra(c)\\
& = & (a\gamma^2)^q+a\gamma^2-\gamma^{q+1}+\Tra(c)=a^q\gamma^2(\gamma^{q-1}-2a)^2+\Tra(c).
\end{array} 
$$
Let $\Tra(c)=k$, with $k\in\Fq^*$. Let us consider two distinct cases:
\begin{itemize}
    \item[-] $\Tra(c)=\Tra(\bar c)$. By (\ref{tracce}) we have that $\Tra(c)=\Tra(\bar c) \iff b=0$.\\ 
So the number of pairs $(0,\bar c)$ are exactly $q^2-q$, because they correspond to all $\bar c\in\FQ$ such that $\Tra(\bar c)\ne 0$.
    \item[-] $\Tra(c)\ne \Tra(\bar c)$. Then $\Tra(\bar c)=a^qb^2+k$.\\
Since $b=2a\gamma-\gamma^q$, then, by considering all possible $\gamma$'s, we obtain $q-1$ distinct $b$'s.\\
In fact, we can consider the function $f:\FQ\rightarrow \FQ$ such that
$f(\gamma)=2a\gamma-\gamma^q$. By  Lemma \ref{corFON}, for any $t\in \I(f)$,
the equation $f(\gamma)=t$ has $q$ distinct solutions.\\
Since we are interested in the case $b\ne 0$, we have $\frac{(q^2-q)}{q}=q-1$ different $b$'s.
\noindent We can note that if $b$ is a solution of the equation $2a\gamma-\gamma^q=0$, then $-b$ is also a solution.\\
\noindent Since we are interested in the pairs $(b^2,\bar c)$,
we note that we have to consider the equation $\Tra(\bar c)=a^qb^2+k$,
so the pairs $(b^2,\bar c)$ are exactly $\frac{q-1}{2}(q^2-q)$.
In fact there are $\frac{q-1}{2}$ distinct $b^2$'s and for any pairs $(b^2, k)$ we have exactly $q$ distinct $\bar c\,$'s. While the possible $k$'s are exactly $q-1$ (because $\Tra(c)\ne 0$).\\
All possible pairs $(b,\bar c)$ are $2\frac{q-1}{2}(q^2-q)=(q-1)(q^2-q)$.
\end{itemize}

\noindent We fix $a$ and we obtain exactly
$(q-1)(q^2-q)+q^2-q=q^2(q-1)$ parabolas of the type $y=ax^2+bx+\bar c$.

In conclusion if $\Delta=0$ and $\Tra(c)\ne 0$, then we have
$q^2(q+1)\frac{q-1}{2}$ parabolas with $2q$ or $0$ intersections.
\end{itemize}
The last type of parabolas cannot be easily counted and so we obtain their number by difference.
\begin{claim} The  number of  parabolas that have $q$ intersections with the Hermitian curve $\He$ is $q^2(q+1)(q^2-q+1)$.
\end{claim}
\begin{proof}
The number of total parabolas is $q^4(q^2-1)$. By summing all parabolas that we already counted we obtain
$$q^2(q+1)\left(2\frac{q-1}{2}+q\frac{q-1}{2}(q-1+q-3)+\frac{q}{2}(q-1+q-3)\right)=$$
$$=q^2(q+1)(q-1+q^2(q-2)).$$
So the number of parabolas  that have $q$ intersections with $\He$ is
$$
q^4(q^2-1)-q^2(q+1)(q-1+q^2(q-2))=
$$
$$
q^2(q+1)(q^2(q-1)-q+1-q^2(q-2))=
q^2(q+1)(q^2-q+1).
$$
\end{proof}

We have proved the following theorems, depending on the two conditions $\Tra(c)=0$ or $\Tra(c) \ne 0$.
\begin{theorem}\label{teo.intTr0}
Let $q$ be odd. A parabola $y=ax^2+c$ with $\Tra(c)= 0$ intersects the Hermitian curve $\He$  in $2q-1, q$ or $1$ points.\\ 
Moreover, we have
\begin{description}
  \item[]$(q+1)\frac{q-1}{2}q^3$ parabolas that intersect $\He$ in $2q-1$ points.
  \item[]$q^2(q+1)(q^2-q+1)$ parabolas that intersect $\He$ in $q$ points.
  \item[]$(q+1)\frac{q-3}{2}q^3$ parabolas that intersect $\He$ in $one$ point.
\end{description}
\end{theorem}

\begin{theorem}\label{teo.intTrn0}
Let $q$ be odd. A parabola $y=ax^2+c$ with $\Tra(c)\ne 0$ intersects the Hermitian curve $\He$ in $2q, q+1,q-1$ or $0$ points.\\ 
Moreover, we have
\begin{description}
  \item[]$q^2(q+1)\frac{q-1}{2}$ parabolas that intersect $\He$ in $2q$ points.
  \item[]$q^3(q+1)(q-1)\frac{q-3}{2}$ parabolas that intersect $\He$ in $q+1$ points.
  \item[]$q^3(q+1)\frac{(q-1)^2}{2}$ parabolas that intersect $\He$ in $q-1$ points.
  \item[]$q^2(q+1)\frac{q-1}{2}$ parabolas that intersect $\He$ in $0$ point.
\end{description}
\end{theorem}
Therefore, by Theorem \ref{teo.intTr0} and Theorem \ref{teo.intTrn0}, 
we obtain the first half of Theorem \ref{teo.principe}.


\subsection{Even characteristics}\label{even}

In this subsection, $q$ is always even.

We claim that it is enough to consider just two special cases:
$y=ax^2$ and $y=ax^2+c$. Before studying these two cases, we consider the following lemma.


%
\begin{lemma}\label{lem:abeven}
Let $x=\alpha^j\beta^i$, with $j=0,\ldots,q$ and $i=0,\ldots,q-2$; then the
values $F_a(\alpha^j\beta^i)$ that are not zero are all the elements of $\Fq^*$.
\end{lemma}
\begin{proof}
Fixing an index $j$, by Lemma \ref{lemmaFa} we have $F_a(\alpha^j \beta^i)=\beta^{2i}F_a(\alpha^j)$. If $F_a(\alpha^j)=0$ we have finished, otherwise $\beta^{2i}F_a(\alpha^j)$ are all elements of $\Fq^*$,
because also $\beta^2$ is a primitive element of $\Fq$.
\end{proof}

We divide the study into two parts.

\begin{itemize}
  \item[$*$] Case $y=ax^2$. We intersect $\He$ with $y=ax^2$ and 
we obtain 
\begin{equation}\label{eq.stel}
x^2(a^q x^{2q-2}-x^{q-1}+a)=0.
\end{equation}
We set $x^{q-1}=t$
and we have to solve the equation $a^q t^2-t+a=0$. Setting $z=ta^q$ we obtain
$$
z^2+z+a^{q+1}=0.
$$
It is known that this equation has solutions in a field of characteristic even
if and only if $\Tr{\FQ}{\F}{a^{q+1}}=0$ (by special case of Artin - Schreier Theorem, see Theorem 6.4 of  \cite{CGC-alg-book-lang02}). 
To show that this latter condition holds, observe first that when 
$\delta\in \Fq$ and $q$ is even, we have $
\Tr{\FQ}{\Fq}{\delta}=0$. Second, observe that
$\Tr{\FQ}{\F}{\cdot}=\Tr{\Fq}{\F}{\Tr{\FQ}{\Fq}{\cdot}}$. Then we can write
$$
\Tr{\FQ}{\F}{a^{q+1}}=\Tr{\Fq}{\F}{\Tr{\FQ}{\Fq}{a^{q+1}}}=\Tr{\Fq}{\F}{0}=0.
$$
We also have $\No(t)=1$, in fact $t^{q+1}=(x^{q-1})^{q+1}=1$. Then we have
$$
z^{q+1}=\No(z)=\No(a^q)=a^{q^2+q}=a^{q+1}=\No(a)
$$
and so the equation becomes $z^2+z+z^{q+1}=0$. Since $t\ne 0,z\ne 0$, then we must have $z^q+z=1.$
We can note that, since $a^{q+1}\in\Fq$, then it is possible to compute its trace
from $\Fq$ to $\F$, and we obtain
{\small{$$
\Tr{\Fq}{\F}{a^{q+1}}=\Tr{\Fq}{\F}{z^{q+1}}=\Tr{\Fq}{\F}{z^2+z}=
z+z^2+z^2+z^4+\ldots+z^{q/2}+z^q=z+z^q.
$$}}
\noindent If it is equal to $0$, we have a contradiction, then there is not any solution $x\in\FQ$.
On the other hand if it is equal to $1$, then we have solutions.

When the solutions exist, since $z^{q+1}=a^{q+1}$,
 a solution $z$ is $a\alpha^{j(q-1)}$, for some $j$,
and the other is $z+1$, which we can write as $a\alpha^{j'(q-1)}$.
From each of these we have the corresponding $t=(\frac{\alpha^j}{a})^{q-1}$ and so the 
$x$'s are $\frac{\alpha^{j+i(q+1)}}{a}=\frac{\alpha^j \beta^i}{a}$ and
$\frac{\alpha^{j'+i(q+1)}}{a}=\frac{\alpha^{j'}\beta^i}{a}$, with $i=0,\ldots,q-2$.


By denoting $A=a^{q+1}$, we summarize the two distinct cases:
\begin{itemize}
  \item[-] If $\Tr{\Fq}{\F}{A}=0$,
then equation (\ref{eq.stel}) has only $one$ solution. On the
other hand, $\Tr{\Fq}{\F}{A}=0$ is satisfied by $q/2$ values for $A$, one of which is $a=0$, which is impossible. So only $\frac{q}{2}-1$ values are actually possible for $A$, each of them having $q+1$ solutions to the equation $a^{q+1}=A$. Therefore, the total number of values for a is   $(\frac{q}{2}-1)(q+1)$.
  \item[-] If $\Tr{\Fq}{\F}{A}=1$,
then equation (\ref{eq.stel}) has $2q-1$ solutions. This happens
  for $\frac{q}{2}$ values of $A$, so the possible values of $a$ are
  $\frac{q}{2}(q+1)$.
\end{itemize}

As in the odd case, we apply the automorphism (\ref{autom}) to the parabolas
of type $y=ax^2$ and we have that distinct automorphisms generate distinct parabolas.
We omit the easy adaption of our earlier proof.\\ 
We have proved the following theorem:
\begin{theorem}\label{teo.intc0}
The Hermitian curve $\He$ and the parabola $y=ax^2$ intersect in either $one$ point or $2q-1$ points.\\ 
Moreover, from the application of (\ref{autom}) to these parabolas, we obtain:
\begin{description}
  \item[]$q^3(\frac{q}{2}-1)(q+1)$ parabolas with $one$ point of intersection with $\He$.
  \item[]$q^3\frac{q}{2}(q+1)$ parabolas with $2q-1$ points of intersection with $\He$.
\end{description}
\end{theorem}

\item[*] Case $y=ax^2+c$ with $\Tra(c)\ne 0$. We consider the equation (\ref{prin}).
We divide the problem into two parts:\\
\begin{itemize}
  \item[-] If $\Tr{\Fq}{\F}{a^{q+1}}=0$, we know that $F_a(x)$ is equal to zero only for $x=0$. 
  If $x\ne 0$, then by Lemma \ref{lem:abeven} if we fix $j$ we have that
  $F_a(x)=F_a(\alpha^j\beta^i)=\beta^{2i}F_a(\alpha^j)$ are all the elements of $\Fq^*$.
  But $j$ can assume $q+1$ distinct values, so any value of $\Fq^*$ can be obtained $q+1$ times.
  So, the equation $F_a(x)=\Tra(c)$ has exactly $q+1$ solutions.\\

  \item[-] If $\Tr{\Fq}{\F}{a^{q+1}}=1$, $F_a(x)=0$ has $2q-1$ solutions.
  So, if we fix an index $j$, the values of $F_a(\alpha^j\beta^i)=\beta^{2i}F_a(\alpha^j)$
  are all equal to zero or are all the elements of $\Fq^*$. There are exactly
  two distinct values of $j$ that give zero, so any non-zero value of $\Fq$ can be obtained $q-1$ times.
  So, the equation $F_a(x)=\Tra(c)$ has exactly $q-1$ solutions.
\end{itemize}

\noindent We apply the automorphism (\ref{autom}) to the parabola $y=ax^2+c$
and we obtain (\ref{parc}). These are all distinct and different from those of Theorem~\ref{teo.intc0},
because the planar intersection of $\He$ and the previous parabolas are different. 
The number of values of $c$ such that $\Tra(c)\ne 0$ is exactly $q^2-q$,
but we must be careful and not count twice the same parabola.
In particular, if two parabolas share $a$ and $b$, then they are in the
same orbit if $\Tra(c)=\Tra(\bar c)$. So we must consider only one of these
for any non-zero value of $\Tra(c)$. These are $q-1$ of these values.\\
Summarizing, we have proved the following theorem:
\begin{theorem}\label{teo.intcn0}
The Hermitian curve $\He$ and the parabola $y=ax^2+c$ with $\Tra(c)\ne 0$ intersect in  either $q+1$ or $q-1$ points.\\ 
Moreover, from the application of (\ref{autom}) to these parabolas, we obtain:
\begin{description}
  \item[]$q^3(\frac{q}{2}-1)(q+1)(q-1)$ parabolas (with $\Tr{\Fq}{\F}{a^{q+1}}=0$)
  with $q+1$ points of intersection with $\He$.
  \item[]$q^3\frac{q}{2}(q+1)(q-1)$ parabolas (with $\Tr{\Fq}{\F}{a^{q+1}}=1$)
  with $q-1$ points of intersection with $\He$.
\end{description}
\end{theorem}

\noindent  By summing all parabolas that we have found in Theorem \ref{teo.intc0} and Theorem~\ref{teo.intcn0}, we obtain
{\small{$$
\begin{array}{l}
   q^3(q+1)\left(\frac{q}{2}-1+
\frac{q}{2}+(q-1)(\frac{q}{2}-1+\frac{q}{2})\right)=\\ 
=q^3(q+1)(q-1)(1+q-1)=
q^4(q^2-1).
\end{array}$$}}
Since this is exactly the total number of the parabolas, this means that we 
considered all parabolas, and so we obtain the second half of Theorem~\ref{teo.principe}.
\end{itemize}

%% file: codHer.tex
The results present in this paper do allow the explicit determination
of at least one weight for some Hermitian codes.\\

We consider a Hermitian code as a special case of affine-variety code.\\

Let $I=\langle y^q+y-x^{q+1},x^{q^2}-x,y^{q^2}-y\rangle\subset\FQ[x,y]$ and let
$R=\FQ[x,y]/I$.  Let $\mathcal{V}(I)=\{P_1,\dots,P_n\}$, where $n=q^3$.
We consider the \textit{evaluation map} defined as follows:
$$
\begin{array}{rcl}
  \phi:R & \longrightarrow & (\FQ)^n\\
  f & \longmapsto & (f(P_1),\dots,f(P_n)).
\end{array}    
$$
We take $L\subseteq R$ generated by
$$\Bmq=\{x^ry^s+I\mid qr+(q+1)s\leq m,\,\,0\leq s\leq q-1,\,\,0\leq r\leq q^2-1\},$$
where $m$ is an integer such that $0\leq m\leq q^3+q^2-q-2$. For simplicity, we also write
$x^ry^s$ for $x^ry^s+I$. 
We have the following affine-variety codes:
$C(I,L)=\mathrm{Span}_{\FQ}\langle\phi(\Bmq)\rangle$
and we denote by $C(m,q)=(C(I,L))^{\perp}$ its dual.
Then the affine-variety code $C(m,q)$ is called the \textit{Hermitian code} with parity-check matrix $H$.
$$
  H=\left(
    \begin{array}{ccc}
      f_1(P_1) & \dots & f_1(P_n)\\
      \vdots & \ddots & \vdots\\
      f_k(P_1) & \dots & f_k(P_n)\\
    \end{array}
  \right)\textrm{ where } \Bmq=\{f_1,\ldots,f_k\}.
$$
The Hermitian codes can be divided in four phases (\cite{CGC-cd-book-AG_HB}), any of them having specific explicit formulas linking their dimension and their distance (\cite{CGC-cd-phdthesis-marcolla}), as in Table~\ref{Tab}.\\

{\scriptsize{
\begin{table}[h]
\begin{center}
\begin{tabular}[h]{|cccc|}
\hline\rowcolor{lightgray}\noalign{\smallskip}
\textbf{Phase } & $\mathbf{m}$ & \textbf{Distance} $\mathbf{d}$ & \textbf{Dimension}\\
\noalign{\smallskip}
\hline
\noalign{\smallskip}
\textbf{1} & {\scriptsize$\begin{array}{c}
  0\leq m\leq q^2-2\\
   m=aq+b\\
  0\leq b\leq a \leq q-1\\
b\ne q-1
\end{array}$} & {\scriptsize$\begin{array}{ll}
  a+1 & a>b\\
  a+2 & a=b
\end{array}\iff d\leq q$} & {\scriptsize$
  q^3-\frac{a(a+1)}{2}-(b+1)$}\\
  & & &\\
\textbf{2} &  {\scriptsize$\begin{array}{c}
  q^2-1\leq m\leq 2q^2-2q-3\\
  m=2q^2-q-aq-b-3\\
  1\leq a\leq q-2\\
  0\leq b\leq q-2
\end{array}$} & {\scriptsize$\begin{array}{ll}
  (q-a)q-b-1 & a\leq b\\
(q-a)q & a > b
\end{array}$} & {\scriptsize$
   n-\frac{q(3q+1)}{2}+aq+b+2$}\\
  & & &\\
\textbf{3} & {\scriptsize$2q^2-2q-2\leq m \leq n-2$} & {\scriptsize$m-q^2+q+2$}& {\scriptsize$n-m+\frac{q(q-1)}{2}-1$}\\
  & & &\\
\textbf{4} & {\scriptsize$\begin{array}{c}
  n-1\leq m \leq n+q^2-q-2\\
  m=n+q^2-q-2-aq-b\\
0\leq b\leq a\leq q-2,
\end{array}$} & {\scriptsize$n-aq-b$}& {\scriptsize$\frac{a(a+1)}{2}+b+1$}\\
\hline
\end{tabular}
\end{center}
\caption{The four phases of Hermitian codes}\label{Tab}
\end{table}}}

In the remainder of this section we focus on the first phase.
This phase can be characterized by the condition $d\leq q$. First-phase Hermitian codes can be either \textit{edge codes} or \textit{corner codes}, as explained below.
\begin{definition}\label{Cornercode}
Let $2\leq d\leq q$ and let $1\leq j\leq d-1$.\\
Let $L_0^d=\{1,x,\dots,x^{d-2}\},L_1^d=\{y,xy,\dots,x^{d-3}y\},\ldots,L_{d-2}^d=\{y^{d-2}\}$.\\
Let $l_1^d=x^{d-1},\ldots,l_j^d=x^{d-j}y^{j-1}$.
\begin{itemize}
  \item[\ding{71}] If $\Bmq=L_0^d\sqcup\dots\sqcup L_{d-2}^d$, then we say that $C(m,q)$ is a
  \textbf{corner code} and we denote it by $\textsf{H}^{\,0}_d$.
  \item[\ding{71}] If $\Bmq=L_0^d\sqcup\dots\sqcup L_{d-2}^d \sqcup\{l_1^d,\ldots,l_j^d\}$,
  then we say that $C(m,q)$ is an \textbf{edge code} and we denote it by $\textsf{H}^{\,j}_d$.
\end{itemize}
\end{definition}
From the formulas in Table \ref{Tab} we have the following theorem.
\begin{theorem}
Let $2\leq d\leq q$, $1\leq j\leq d-1$. Then
$$d(\textsf{H}^{\,0}_d)=d(\textsf{H}^{\,j}_d)=d,\quad
\dim_{\FQ}(\textsf{H}^{\,0}_d)=n-\frac{d(d-1)}{2},\,\,
\dim_{\FQ}(\textsf{H}^{\,j}_d)=n-\frac{d(d-1)}{2}-j$$
\end{theorem}
\noindent In other words, all $\phi(x^ry^s)$ are linearly independent (i.e. $H$ has maximal rank)
and for any distance $d$ there are exactly $d$ Hermitian codes (one corner code and $d-1$ edge codes).\\

A result present in \cite{CGC-cod-art-marcolla2012hermitian} related to intersection between parabola and $\He$ is in the following theorem.

\begin{theorem}
The number of words of weight $4$ of a corner code $\textsf{H}^{\,0}_3$ is:
{\small{$$A_4=\frac{1}{4}\left(\binom{q^3}{3}(q+1)-q^2\binom{q+1}{3}
(3q^3+2q^2-8)\right)(q-1)(q^3-3).$$}}
\noindent The number of words of weight $4$ of an edge code $\textsf{H}^{\,1}_3$ is:
{\small{$$A_{4}=q^2\binom{q}{4}(q^4-4q^2+3)+\frac{q^4(q^2-1)^2(q-1)^2}{8}+
(q^2-1)\sum_{k=4}^{2q}N_k\binom{k}{4}.$$}}

\noindent Where $N_k$ is the number of \textbf{parabolas} and non-vertical lines that intersect $\He$ in exactly $k$ points.\\
The number of words of weight $4$ of an edge code $\textsf{H}^{\,2}_3$ is:
{\small{$$A_{4}=q^2(q-1)\binom{q+1}{4}(2q^3-3q^2-4q+9).$$}}%
\begin{proof}
See Theorem 4.13 of \cite{CGC-cod-art-marcolla2012hermitian}.
\end{proof}
\end{theorem}

Since the intersections between $\He$ and the lines are easy to compute, 
to determine $N_k$ in previous theorem it is enough to apply Theorem~\ref{teo.principe}.

%% file: conc.tex
Apparently, there are two  natural generalizations of our work:
\begin{itemize}
\item
The first is to investigate the planar intersection of $\He$ with other conics.
Unfortunately, this is not so easy as it seems.
The case with parabolas is manageable because they intersect the curve in a
tangency double point at infinity. In the general conics case, we will not have
this help from the geometry and we will not be able to replicate
many explicit computations we have done in our lemmas.
Therefore, if someone wants to investigate the general case, he will need some
extra (non-trivial) ideas.
 \item
The second is to investigate the planar intersections of parabolas and
other curves.
The natural candidates are the norm-trace curves \cite{CGC-cd-art-Geil-1}, which share many properties with the Hermitian curve.
\end{itemize}
In both cases, the investigation is not only interesting in itself,
but it is likely to shed light on the weight distribution of some affine-variety codes.

%% file: ack.tex
This work was partially presented in 2012 at the PhD School on \Gr\
bases, curves, codes and cryptography (Trento) and in 2013 at Effective Methods in Algebraic Geometry, MEGA \cite{CGC-alg-talk-parabole}.\\

Previous results were present in the first author's PhD thesis \cite{CGC-cd-phdthesis-marcolla}.\\

The first two authors would like to thank their supervisor, the third author.\\

For interesting discussions, the authors would like to thank: M. Giulietti, T. Mora and M. Pizzato.\\
%